\documentclass[11pt,oneside]{article}
\usepackage{amssymb}
\usepackage{amsmath}
\usepackage{array}
\usepackage{enumerate}
\usepackage{graphicx}
\usepackage{longtable}
\usepackage{amsmath}
\usepackage{amsmath, nccmath}
\usepackage{geometry}
\usepackage{amsthm}
\usepackage{mathrsfs}
\usepackage{fancyhdr}
\usepackage{pstricks}
\usepackage{quotes}
\usepackage{graphicx}
\usepackage{lineno}
\newtheorem{thm}{Theorem}
\newtheorem{lem}[thm]{Lemma}
\theoremstyle{definition}

\newtheorem{coro}[thm]{Corollary}

\newtheorem{rem}[thm]{Remark}

\newtheorem{pbm}{Problem}

\begin{document}
\title{\bf A Reduction of the Reconstruction Conjecture using Domination and Vertex Pair Parameters}
\author{\normalsize  J. Antony Aravind and  S. Monikandan\\ [-0.1cm]
\normalsize Department of Mathematics\\[-0.1cm]
\normalsize Manonmaniam Sundaranar University\\[-0.1cm]
\normalsize Tirunelveli \\[-0.1cm]
\normalsize Tamilnadu, INDIA\\[-0.1cm]
\normalsize  antonyaravind123@gmail.com, monikandans@msuniv.ac.in \vspace{-.2cm}}
\date{}
\maketitle

\begin{abstract}
A graph is \emph{reconstructible} if it is determined up to isomorphism from the collection of all its one-vertex-deleted subgraphs, known as the \emph{deck} of $G$. The \emph{Reconstruction Conjecture} (RC) posits that every finite simple graph with at least three vertices is reconstructible.
In this paper, we prove that the class of graphs with domination number $\gamma(G)=2$ is recognizable from the deck $D(G)$. We also establish a new reduction of the RC: it holds if and only if all $2$-connected graphs $G$ with $\gamma(G)=2$ or $\operatorname{diam}(G)=\operatorname{diam}(\overline{G})=2$ are reconstructible. To aid reconstruction, we introduce two new parameters: $dv(G,k_1,k_2,k_3)$, which counts the number of non-adjacent vertex pairs in $G$ with $k_1$ common neighbours, $k_2$ neighbours exclusive to the first vertex, and $k_3$ exclusive to the second; and $dav(G,k_1,k_2,k_3)$, defined analogously for adjacent pairs. For connected graphs with at least $12$ vertices and $\gamma(G)\geq 3$, we show these parameters are reconstructible from $D(G)$ via recursive equations and induction. Finally, we prove that $k$-geodetic graphs of diameter two with $\gamma(G),\gamma(\overline{G})\geq 3$ are reconstructible under conditions where a vertex degree matches the size of a specific subset derived from these parameters.

\end{abstract}



\footnote{Research is supported by  NBHM, Department of Atomic Energy, Govt. of India, Mumbai through a Major Research project. Sanction order No. 02011/14/2022/NBHM(R.P)/$\text{R\&D}$ II/10491.}

\section{Introduction}\label{intro}
In this paper, we explore finite, simple, and undirected graphs. For any graph $G,$ we denote its set of vertices as $V(G)$ and its set of edges as $E(G).$ Terminology not explicitly defined here follows the standard conventions outlined in \cite{w}. The degree of a vertex $v$ in $G,$ $deg_{G}v$ (or simply $deg~v$), represents the number of edges incident to $v.$ A path from vertex $u$ to vertex $v$ is termed $u,v$-path, and if such a path exists, the distance $d_{G}(u,v)$ or simply $d(u,v)$  is the length of a shortest $u,v$-path. The \emph{diameter} of $G,$ denoted $\mathrm{diam}(G),$ is the maximum distance between any two vertices in $V(G).$ The \emph{connectivity} of $G,$ written $\kappa(G),$ is the smallest number of vertices whose removal either disconnects $G$ or leaves it with a single vertex.  The \emph{neighbourhood} $N_{G}(v)$ or simply $N(v)$ consists of all vertices adjacent to  $v,$ and $N[v]=N(v)\cup \{v\}.$ For a subset $S\subseteq V(G),$  the subgraph induced by $S$ is denoted by $G[S].$ The notation $[n]$ represents  the set $\{1,2,...,n\},$  and a set $A$ is \emph{finite} if it can be bijectively mapped to $[n]$ for some $n \in \mathbb{N}\cup \{0\},$  with its size. This $n$ is the \emph{size} of $A,$ written $|A|.$ In a graph $G,$ a set $S \subseteq V(G)$ is a \emph{dominating set} if every vertex outside $S$ is adjacent to at least one vertex in $S,$ and the \emph{domination number} $\gamma(G)$ is the size of the smallest such set.

\indent A \emph{vertex-deleted subgraph} (or \emph{card}) $G-v$ of $G$ is the unlabeled subgraph formed  by removing  a vertex $v$ and all edges incident to it from $G$. The \emph{deck} of $G$ is the collection of all such cards, denoted by $\mathscr{D}(G).$  A graph $H$ is a \emph{reconstruction} of $G$ if it shares the same deck as $G.$ A graph is \emph{reconstructible} if every graph with the same deck  is isomorphic to it. A family $\mathcal{F}$  of graphs is \emph{recognizable} if all reconstructions of  a graph in $\mathcal{F}$ remains in $\mathcal{F},$ and \emph{weakly reconstructible} if all reconstructions of  a graph in $\mathcal{F}$ that stay within $\mathcal{F}$  are isomorphic to it. A family $\mathcal{F}$ of graphs is \emph{reconstructible} if it is both recognizable and weakly reconstructible, meaning every graph  in $\mathcal{F}$ is reconstructible. A parameter $p$ of graphs is reconstructible if its value is the same across all  reconstructions of a graph. The \emph{Reconstruction Conjecture} (RC) posits that all graphs with at least three vertices are reconstructible.  For comprehensive reviews of this problem, see the surveys in \cite{b}, \cite{bh}, \cite{h} and \cite{l}. 

\indent In 1988, Yang Yongzhi \cite{y} demonstrated that the RC holds if and only if every 2-connected graph is reconstructible. The Digraph Reconstruction Conjecture (DRC) was disproved by Stockmeyer\cite{pks},  implying that a proof of the RC must hinge on properties unique to undirected graphs, not extensible to directed graphs.  One such property relates to distances in the graph's complement, supported by two established results: a graph $G$ is reconstructible if and only if  its complement $\overline{G}$ is reconstructible, and if $\mathrm{diam}(G)>3,$ then $\mathrm{diam}(\overline{G})<3.$" Leveraging these, Gupta et al. \cite{g} in 2003 showed that the RC is true if and only if all connected graphs $G$ with $\text{diam}(G)=2$ or both $\mathrm{diam}(G)=3$ and $\mathrm{diam}(\overline{G})=3$ are reconstructible.  Ramachandran and Monikandan \cite{r} refined this further, proving that the RC holds if and only if all $2$-connected graphs $G$ satisfying $\mathrm{diam}(G)=2$ or $\mathrm{diam}(G)=\mathrm{diam}(\overline{G})=3$ are reconstructible. Devi Priya and Monikandan \cite{d} provided an additional reduction, showing that all distance-hereditary 2-connected graphs $G$ with $\mathrm{diam}(G)=2$ or $\mathrm{diam}(G)=\mathrm{diam}(\overline{G})=3$ are reconstructible.

Bondy and Hemminger \cite{b} proposed two strategies to prove the RC: reconstructing classes of graphs to eventually encompass all graphs, or reconstructing  graph parameters. Pursuing the latter,   Gupta et al. \cite{g} defined two parameters $pv(G,i)$ (for$i \in [0,n-2]$), counts the number of pairs of non adjacent vertices in $G$ with exactly $i$ paths of length two between them, and $pav(G,i)$ (for $i \in [0,n-2]$) counts the same for adjacent pairs. They proved the next theorem. 

\begin{thm}\label{1}\normalfont{\cite{g}}
Parameters $pv(G,i)$ and $pav(G,i) \forall i \in [0,n-2]$ are reconstructible.
\end{thm}

{\it Since graphs with up to eleven  vertices are  known to be reconstructible \cite{bd}, this paper focuses on  graphs with at least twelve vertices.} In this paper, we introduce two new parameters, $dv(G,k_1,k_2,k_3)$ and $dav(G,k_1,k_2,k_3),$ where  $dv(G,k_1,k_2,k_3)$ counts pairs of non-adjacent vertices $x$ and $y$ in $G$ with $k_1$ vertices adjacent to both $x$ and $y,$ $k_2$ vertices adjacent only to $x,$ and $k_3$ vertices adjacent only to $y.$  But $dav(G,k_1,k_2,k_3)$ counts pairs of adjacent vertices $x$ and $y$ in $G$ with $k_1$ vertices adjacent to both $x$ and $y,k_2$ vertices adjacent only to $x~(y ~ is ~ excluded ~)$ and $k_3$ vertices adjacent only to $y ~(x ~ is ~ excluded ~).$ In Section 2, we establish a reduction of the Reconstruction Conjecture (RC) based on the domination number and diameter. To determine the number of common neighbours and degrees for vertex pairs (adjacent or non-adjacent) from the deck $\mathscr{D}(G)$, the domination number $\gamma(G)$ must exceed two: if $\gamma(G)=1$, then $G$ has a universal vertex $x$ with $\deg_G x = n-1$, rendering it reconstructible; if $\gamma(G)=2$, then two vertices $u,v$ (adjacent or not) dominate $V(G)\setminus\{u,v\}$, but their degrees are indeducible from $\mathscr{D}(G)$, as no card contains both alongside the rest of $G$. Thus, in Sections 3 and 4, we characterize the structure of graphs with $\gamma(G)=2$ and prove that this class is recognizable from $\mathscr{D}(G)$. For graphs with $\gamma(G)>2$, Section 5 employs the new parameters $dv(G,k_1,k_2,k_3)$ and $dav(G,k_1,k_2,k_3)$ to count non-adjacent and adjacent vertex pairs (incorporating degrees) and shows they are reconstructible from $\mathscr{D}(G)$. Finally, in Section 6, we reconstruct $k$-geodetic graphs $G$ with $diam(G)=diam(\overline{G})=2$ under specific parameter conditions.

The following two theorems are well-known in graph reconstruction.
\begin{thm}\label{t1}
$G$ is reconstructible if and only if  $\overline{G}$ is reconstructible.
\end{thm}

\begin{thm}\label{t2}
Disconnected graphs are reconstructible.
\end{thm}

\begin{lem}\label{L3}
If $G$ is connected and  $diam(G) \geq 3,$ then $\gamma(\overline{G})=2.$
\end{lem}
\begin{proof}
If $\text{diam}(G) \geq 3,$ there exist non-adjacent vertices $u$ and $v$ in $G$ with $d_{G}(u,v) \geq 3.$ Therefore, $u$ and $v$ have no common neighbour in $G.$ Hence, in $\overline{G},$ vertices $u$ and $v$ are adjacent and every other vertex is adjacent to at least one of them. Therefore $\{u,v\}$ is a dominating set of $\overline{G}$, so $\gamma(\overline{G}) \leq 2.$ Since $G$ contains no isolated vertex,  $\overline{G}$ contains no vertex adjacent to all other vertices. Therefore $\gamma(\overline{G})\geq 2,$ so $\gamma(\overline{G})=2.$
\end{proof}

\section{A reduction of RC}
\indent \indent In this section, we focus on specific classes of connected graphs defined by their domination number $(\gamma(G))$ and diameter $(\text{diam}(G))$, aiming to simplify the conjecture's scope.

\begin{thm}\label{21}
All connected graphs are reconstructible if and only if all connected graphs $G$ with  $\gamma(G)=2$ or 
$\text{diam}(G)=\text{diam}(\overline{G})=2$ are reconstructible.

\begin{proof}
Necessity  is obvious. For sufficiency, assume  all connected graph $G$ such that either $\gamma(G)=2$ or $"\text{diam}(G)=\text{diam}(\overline{G})=2, \gamma(G) \geq 3 ~ and ~ \gamma(\overline{G}) \geq 3"$ are reconstructible. Since graphs of diameter 1 as well as  2  are recognizable, we proceed by cases based on $\text{diam}(\overline{G}).$ \\
 If $\text{diam}(\overline{G})=1,$ $\overline{G}$ is a complete graph, so it is reconstructible. \\ Suppose $\text{diam}(\overline{G})=2.$ Now, if $\text{diam}(G)=1,$ then $G \cong K_n,$  so $G$ is reconstructible. If $\text{diam}(G)=2,$ then $G$ is reconstructible by our assumption. If $\text{diam}(G) \geq 3,$ then $\gamma(\overline{G})=2,$ by Lemma \ref{L3} and $\overline{G}$ is reconstructible by our assumption. \\ So, assume that $\text{diam}(\overline{G}) \geq 3.$
 Now, by Lemma \ref{L3},  $\gamma(G)=2$ and hence $G$ is reconstructible by our assumption.
\end{proof} 
\end{thm}
  
\begin{thm}(Yongzhi\cite{y}) \label{22}
Every connected graph is reconstructible if and only if every $2$-connected graph is reconstructible.
\end{thm} 

\begin{thm}
The RC is true if and only if all $2$-connected graphs $G$ such that $\gamma(G)=2$ or $\text{diam}(G)=\text{diam}(\overline{G})=2,$ are reconstructible.
\begin{proof}

This follows from prior results:
\begin{itemize}
        \item {\bf Disconnected Graphs}: Theorem \ref{t2} confirms that all disconnected graphs are reconstructible.
        \item  {\bf Connected Graphs}: Theorem \ref{21} states that all connected graphs are reconstructible if and only if the specified classes $\gamma(G)=2$ or $\text{diam}(G)=\text{diam}(\overline{G})=2$ are reconstructible.
        \item {\bf Link to RC:} The RC applies to all graphs with at least three vertices. Since graphs are either connected or disconnected, and disconnected cases are settled, the RC's validity hinges on connected graphs. Theorem \ref{21} reduces this to the specified classes, so the RC holds if and only if these classes are reconstructible. \end{itemize} \end{proof} \end{thm} 

\section{Graphs $G$ with $\gamma(G)=2$}
\indent \indent For a graph $G$ with $\gamma(G)=2,$ we have two vertices $u$ and $v$ forming a dominating set. These vertices may or may not be adjacent to each other in $G$ -- their adjacency is not fixed. The vertex set of $G,$ denoted $V(G),$ is partitioned into four subsets as follows:
\begin{itemize}
\item $\{u,v\}$: The two dominating vertices
\item $X_u$: the set of vertices adjacent only to $u$ and not to $v.$
\item $X_v$: the set of vertices adjacent only to $v$ and not to $u.$
\item $X_{uv}$: the set of vertices adjacent to both $u$ and $v.$
\end{itemize}
It is stated that: $V(G)=\{u,v\} \cup X_u \cup X_v \cup X_{uv}.$\\
This partition implies that every vertex in  $G$ is either $u,v$ or belongs to one of $X_u$, $X_v$, or $X_{uv},$ based on its adjacency to $u$ and $v.$ Since $\{u,v\}$ is a dominating set, every vertex in $V(G)-\{u,v\}$ must be adjacent to at least one of $u$ or $v$. Let us verify this:
\begin{itemize}
\item Vertices in $X_u$ are adjacent to $u$ (and not to $v$), so they are dominated by $u.$
\item Vertices in $X_v$ are adjacent to $v$ (and not to $u$), so they are dominated by $v.$
\item Vertices in $X_{uv}$ are adjacent to both $u$ and $v$, so they are dominated by both. 
\end{itemize}

This confirms that every vertex outside $\{u,v\}$ is adjacent to at least one of $u$ or $v,$ satisfying $\gamma(G)=2$. The vertices  $u$ and $v$
 themselves are in the dominating set, so they do not need to be dominated by another vertex. Thus, the domination condition holds regardless of whether $u$ and $v$ are adjacent in 
$G.$

\begin{figure}[h]
\centering
\scalebox{1} 
{
\begin{pspicture}(0,-2.5)(11,2.5)
\psline[linewidth=0.04cm,doubleline=true,doublesep=0.06](0,0)(2.25,1)
\psline[linewidth=0.04cm,doubleline=true,doublesep=0.06](0,0)(1.75,-1)
\psline[linewidth=0.04cm,doubleline=true,doublesep=0.06](5,0)(3.25,-1)
\psline[linewidth=0.04cm,doubleline=true,doublesep=0.06](5,0)(2.75,1)
\psdots[dotsize=0.18](0,0)
\psdots[dotsize=0.18](5,0)
\psline[linewidth=0.04cm,linestyle=dashed,dash=0.16cm 0.16cm](0,0)(3.25,-0.5)
\psline[linewidth=0.04cm,linestyle=dashed,dash=0.16cm 0.16cm](5,0)(1.75,-0.5)
\psellipse[linewidth=0.04,dimen=outer](2.5,1)(1,0.5)
\pscircle[linewidth=0.04,dimen=outer](1.75,-1){0.5}
\pscircle[linewidth=0.04,dimen=outer](3.25,-1){0.5}
\rput(-0.25,0){$u$}
\rput(5.25,0){$v$}
\rput(2.5,1.75){$X_{uv}$}
\rput(1.75,-1.75){$X_u$}
\rput(3.25,-1.75){$X_v$}
\rput(2.50,-2.25){$G$}

\psline[linewidth=0.04cm,linestyle=dashed,dash=0.16cm 0.16cm](6,0)(8.25,1)
\psline[linewidth=0.04cm,linestyle=dashed,dash=0.16cm 0.16cm](11,0)(8.75,1)
\psline[linewidth=0.04cm,doubleline=true,doublesep=0.06](6,0)(7.75,-1)
\psline[linewidth=0.04cm,doubleline=true,doublesep=0.06](11,0)(9.25,-1)
\psdots[dotsize=0.18](6,0)
\psdots[dotsize=0.18](11,0)
\psline[linewidth=0.04cm,linestyle=dashed,dash=0.16cm 0.16cm](6,0)(9.25,-0.5)
\psline[linewidth=0.04cm,linestyle=dashed,dash=0.16cm 0.16cm](11,0)(7.75,-0.5)
\psellipse[linewidth=0.04,dimen=outer](8.5,1)(1,0.5)
\pscircle[linewidth=0.04,dimen=outer](7.75,-1){0.5}
\pscircle[linewidth=0.04,dimen=outer](9.25,-1){0.5}
\rput(5.75,0){$u$}
\rput(11.25,0){$v$}
\rput(8.5,1.75){$X_{uv}$}
\rput(7.75,-1.75){$X_v$}
\rput(9.25,-1.75){$X_u$}
\rput(8.5,-2.25){$\overline{G}$}
\end{pspicture} 
}
\caption{Structure of a graph $G \in \mathscr{H}$ and its complement}
\label{f1}
\end{figure}
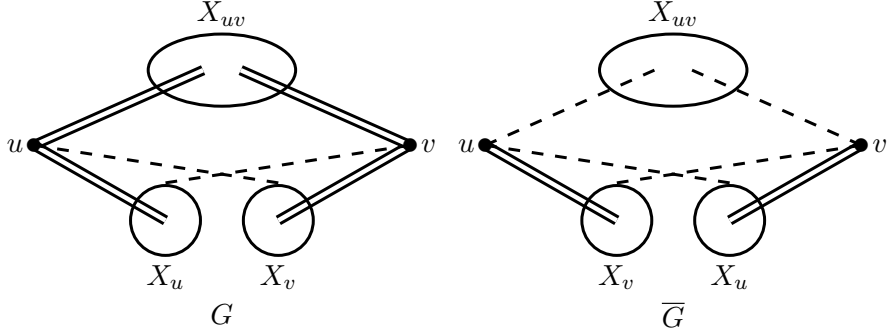
Observe that in the figure described above,  $u$ and $v$ represent individual vertices, while $X_u,X_v$ and $X_{uv}$ represent subsets of vertices. A double line indicates that all possible edges exist between the connected entities, a single dashed line signifies the absence of any edges, and the lack of a line between two vertex subsets or between two vertices implies that edges between them are optional$-$they may or may not be present. 

\section{{ Graphs with domination number two are recognizable}}
 \indent \indent Since graphs with a universal vertex and graphs of order at most eleven are reconstructible, we consider only graphs of order at least twelve and domination number at least two. We define $\mathscr{H}=\{G| \gamma(G)=2\},$ the family of graphs with domination number 2, and partition it into three disjoint subfamilies based on diameter: \\
$\mathscr{H}_1=\{G| \gamma(G)=2 ~ \text{and} ~ diam(\overline{G}) = 1\}.$\\
 $\mathscr{H}_2=\{G| \gamma(G)=2 ~ \text{and} ~ diam(\overline{G}) = 2\}.$\\
 $\mathscr{H}_3=\{G| \gamma(G)=2 ~ \text{and} ~ diam(\overline{G}) \geq 3\}.$

{ \begin{lem}\label{l1}
$\mathscr{H}_1 = \emptyset.$
\end{lem}
\begin{proof}
Suppose there exists a graph $G \in \mathscr{H}_1.$ Then $\text{diam}(\overline{G})=1,$ meaning $\overline{G}$ is a complete graph. Consequently, $G \cong \overline{K_n}$ for some $n.$ But the condition $\gamma(G)=2$ forces $n=2,$ which we already excluded for consideration. Therefore, $\mathscr{H}_1 = \emptyset.$ 
\end{proof}}

\begin{lem}\label{2}
For a graph  $G$  with $\text{diam}(\overline{G})=2,$  $\gamma(G)=2$ if and only if $pav(\overline{G},0)>0.$
\end{lem}
\begin{proof}
{\it Necessity:} Assume $\gamma(G)=2.$ Let $\{u, v\}$ be a dominating set of $G$. Then $uv \notin E(G)$ as otherwise  $d_{\overline{G}}(u,v) \geq 3,$  contradicting the graph $\overline{G}$ having $\text{diam}(\overline{G})=2.$ Now, in $\overline{G}$, the vertices $u$ and $v$ are adjacent and none of the other vertices is a common neighbour of $u$ and $v$.   
This implies $pav(\overline{G},0)>0.$\\
{\it Sufficiency:} Assume $pav(\overline{G},0)>0.$ Then there exists a pair of adjacent vertices $x$ and $y$ in $\overline{G}$ such that every other vertex is either adjacent to exactly one of them or none of them. Therefore, in $G,$ the vertices $x$ and $y$ are non adjacent  and every other vertex is adjacent to at least one of them. Hence the set $\{x, y\}$ is a dominating set of  $G,$ so $\gamma(G)\leq 2.$ But the condition $\text{diam}(\overline{G})=2$  forces $\gamma(G)= 2.$
\end{proof}

\begin{lem}\label{l2}
The class $\mathscr{H}_2$ is recognizable. 
\end{lem} 
\begin{proof}  
From the given deck $\mathscr{D}(G)$ of a graph $G,$ form the deck  $ \mathscr{D}(\overline{G})$ (Theorem \ref{t1}) and compute the reconstructible parameters  $pv(\overline{G},0)$ and $pav(\overline{G},0)$ (Theorem \ref{1}). If $pv(\overline{G},0)=0,$ then $\text{diam}(\overline{G})=2.$ Also, if $pav(\overline{G},0)>0,$ then by Lemma \ref{2},$\gamma(G)=2$ and placing $G$ in  $\mathscr{H}_2.$ Thus, $\mathscr{H}_2$ is recognizable. 
\end{proof}

\begin{lem}\label{14}
The class $\mathscr{H}_3$ is recognizable.
\end{lem}
\begin{proof}
From the deck $\mathscr{D}(\overline{G})$ of $\overline{G},$ compute $pv(\overline{G},0),$ the number of non-adjacent vertex pairs with no common neighbours, which is reconstructible (analogous to Theorem \ref{1}). If  $pv(\overline{G},0)>0,$ then $\text{diam}(\overline{G}) \geq 3$ (since a path of length at least three exists between such pairs). By Lemma \ref{L3},  $\gamma(G)=2,$ so $G \in \mathscr{H}_3,$ making $\mathscr{H}_3$ recognizable.
\end{proof}

\begin{thm} \label{HR}
The class $\mathscr{H}$ is recognizable.
\begin{proof}
 We procced by three cases depending on $\text{diam}(\overline{G}).$  Whether $\text{diam}(\overline{G})$ has value 1, 2 or above 2 can be determined from the deck of $\overline{G}.$ 
\begin{itemize}
\item If $\text{diam}(\overline{G})=1,$ then $\overline{G} \cong K_n,$ and $G$ is disconnected, which is reconstructible (Theorem \ref{t2}). Since $\gamma(\overline{K_n}) =n>2,~ G \notin \mathscr{H}.$
\item If $\text{diam}(\overline{G})=2,$ compute $pav(\overline{G},0)$ (Theorem \ref{1}). If $pav(\overline{G},0)>0,$ then $\gamma(G)=2$ (Lemma \ref{2}), so $ G \in \mathscr{H}_2 \subset \mathscr{H}$ ;  if $pav(\overline{G},0)=0,$ then  $\gamma(G) \neq 2,$ so $ G \notin \mathscr{H}.$
 \item If $\text{diam}(\overline{G})\geq 3,$ then $\gamma(G)=2$ (Lemma \ref{14}), and $ G \in \mathscr{H}_3 \subset \mathscr{H}.$
\end{itemize}
 Thus, $\mathscr{H}$ is recognizable. 

\end{proof}
\end{thm}

\section{Counting Degrees}

 \indent \indent Let $dv(G,k_1,k_2,k_3)$ represent the number of pairs of non-adjacent vertices $x$ and $y$ in $G$ with $k_1$ vertices adjacent to both $x$ and $y,$ $k_2$ vertices adjacent only to $x,$ and $k_3$ vertices adjacent only to $y.$ Let $dav(G,k_1,k_2,k_3)$ represent  the number of pairs of adjacent vertices $x$ and $y$ in $G$ with $k_1$ vertices adjacent to both $x$ and $y,k_2$ vertices adjacent only to $x~(y ~ is ~ excluded ~)$ and $k_3$ vertices adjacent only to $y ~(x ~ is ~ excluded ~).$

\begin{thm}\label{t3}
For a connected graph $G$ on $n$ vertices and $k_2 \neq k_3,$ 
\begin{multline} 
\sum_{i=1}^{n}{dv(G-v_i,k_1,k_2,k_3)}=(n-2-k_1-k_2-k_3)dv(G,k_1,k_2,k_3)\\+(k_1+1)dv(G,k_1+1,k_2,k_3)
+(k_2+1)dv(G,k_1,k_2+1,k_3)+(k_3+1)dv(G,k_1,k_2,k_3+1). 
\end{multline}
\end{thm}
\begin{proof}
To prove this, we evaluate the left-hand side, $\sum_{i=1}^{n}{dv(G-v_i,k_1,k_2,k_3)},$ by counting how pairs of non-adjacent vertices with specific neighbour properties in $G$ contribute to $dv(G-v_i,k_1,k_2,k_3)$ across all cards $G-v_i.$ Recall that $dv(G,k_1,k_2,k_3)$ denotes the number of non-adjacent vertex pairs $\{u,v\}$ in $G$  with $k_1$ vertices adjacent to both $u$ and $v,$ $k_2$ vertices adjacent only to $u,$ and $k_3$ vertices adjacent only to $v.$ We consider four types of pairs in $G$ and their contributions when a vertex is deleted:
\begin{enumerate}
\item[(1)] Pairs with $(k_1,k_2,k_3)$ in $G$:
 Let $\{u,v\}$ be a pair of non-adjacent vertices in $G$ where $k_1$ vertices adjacent to both $u$ and $v,$ $k_2$ vertices adjacent only to $u,$ and $k_3$ vertices adjacent only to $v.$ In the deck $\mathscr{D},$ consider the cards:

\begin{enumerate}
\item[{\bf \LARGE .}] $G-u$ and $G-v:$ The pair  $\{u,v\}$ is absent (since one vertex is deleted).
\item[{\bf \LARGE .}] Cards $G-w$ where $w$ is one of the $k_1+k_2+k_3$  vertices adjacent to $u$ or $v:$ Deleting $w$ reduces the count of neighbours, so the pair $\{u,v\}$ no longer has exactly $k_1,k_2,k_3$ neighbours in $G-w$, and thus does not contribute to $dv(G-v_i,k_1,k_2,k_3)$. 
\item[{\bf \LARGE .}] Remaining cards: There are $n$ vertices total, and we exclude $u$, $v,$ and the $k_1+k_2+k_3$ neighbours, leaving $n-(2+k_1+k_2+k_3)$ cards ( noting $u$ and $v$ are distinct, and $k_2 \neq k_3$ ensures asymmetry). In each of these, $\{u,v\}$ remains a non-adjacent pair with $k_1,k_2,k_3$ unchanged, contributing 1 to $dv(G-v_i,k_1,k_2,k_3)$. Thus, each such pair contributes to $n-(2+k_1+k_2+k_3)$ cards, and the total contribution is $(n-2-k_1-k_2-k_3)dv(G,k_1,k_2,k_3).$
\end{enumerate}
\item[(2)] Pairs with $(k_1+1,k_2,k_3)$ in $G$:\\
Consider a pair $\{u,v\}$ in $G$ with $k_1+1$ common neighbours, $k_2$ vertices adjacent only to $u,$ and $k_3$ vertices adjacent only to $v.$ In $G-w,$ where $w$ is one of the $k_1+1$ common neighbours:
\begin{enumerate}
\item[{\bf \LARGE .}] Deleting $w$ reduces the common neighbours to $k_1,$ while $k_2$ and $k_3$ remain unchanged, so $\{u,v\}$ becomes a $(k_1,k_2,k_3)$-pair in $G-w,$ contributing 1 to $dv(G-w,k_1,k_2,k_3).$
\item[{\bf \LARGE .}] In the remaining $n-(k_1+1)$ cards, the pair either disappears (if $u$ or $v$ is deleted) or has $k_1+1$ common neighbours, not $k_1.$ There are $k_1+1$ such cards, so the total contribution is $(k_1+1)dv(G, k_1+1,k_2,k_3).$
\end{enumerate}

\item[(3)] Pairs with $(k_1,k_2+1,k_3)$ in $G$:\\
Take a pair $\{u,v\}$ with $k_1$ common neighbours, $k_2+1$ vertices adjacent only to $u,$ and $k_3$ vertices adjacent only to $v,$  consider $G-w$ where $w$ is one of the $k_3+1$ vertices adjacent only to $v:$
\begin{enumerate}
\item[{\bf \LARGE .}] Deleting $w$ reduces the count to $k_2$ vertices adjacent only to $u,$ while $k_1$ and $k_3$ stay the same, making  $\{u,v\}$ a $(k_1,k_2,k_3)$-pair in $G-w.$
\item[{\bf \LARGE .}] In other cards, the pair either vanishes or retains $k_2+1.$ With $k_2+1$ such cards, the contribution is $(k_2+1)dv(G, k_1,k_2+1,k_3).$
\end{enumerate}

\item[(4)] Pairs with $(k_1,k_2,k_3+1)$ in $G$:\\
For a pair $\{u,v\}$ with $k_1$ common neighbours, $k_2$ vertices adjacent only to $u,$ and $k_3+1$ vertices adjacent only to $v:$
\begin{enumerate}
\item[{\bf \LARGE .}] Deleting $w$ adjusts the count to $k_3,$ with $k_1$ and $k_2$ unchanged, so $\{u,v\}$ fits $(k_1,k_2,k_3)$ in $G-w.$
\item[{\bf \LARGE .}] Elsewhere, the pair either disappears or keeps $k_3+1.$ This occurs in $k_3+1$ cards, yielding $(k_3+1)dv(G, k_1,k_2,k_3+1).$
\end{enumerate}
\end{enumerate}
Summing these contributions, no other pairs in $G$ produce $k_1,k_2,k_3$ in $G-v_i$ under single vertex deletion (since $k_2 \neq k_3$ ensures distinct roles), so:
\begin{multline} \sum_{i=1}^{n}{dv(G-v_i,k_1,k_2,k_3)}=(n-2-k_1-k_2-k_3)dv(G,k_1,k_2,k_3)\\+(k_1+1)dv(G,k_1+1,k_2,k_3)
+(k_2+1)dv(G,k_1,k_2+1,k_3)+(k_3+1)dv(G,k_1,k_2,k_3+1). 
\end{multline}

\end{proof}

\begin{coro}\label{c1}
When $k_2=k_3=k,$ the equation from Theorem \ref{t3} simplifies to: 
\begin{multline} 
\sum_{i=1}^{n}{dv(G-v_i,k_1,k,k)}=(n-2-k_1-2k)dv(G,k_1,k,k)+(k_1+1)dv(G,k_1+1,k,k)
+\\(k+1)dv(G,k_1,k,k+1). 
\end{multline}
\end{coro}
\begin{proof}
Substitute $k_2=k_3=k$ into Theorem \ref{t3}'s equation. The original terms are:
\begin{multline} \sum_{i=1}^{n}{dv(G-v_i,k_1,k_2,k_3)}=(n-2-k_1-k_2-k_3)dv(G,k_1,k_2,k_3)\\+(k_1+1)dv(G,k_1+1,k_2,k_3)
+(k_2+1)dv(G,k_1,k_2+1,k_3)+(k_3+1)dv(G,k_1,k_2,k_3+1). 
\end{multline}

With $k_2=k_3=k,$ this becomes:

\begin{multline} 
\sum_{i=1}^{n}{dv(G-v_i,k_1,k,k)}=(n-2-k_1-k-k)dv(G,k_1,k,k)+(k_1+1)dv(G,k_1+1,k,k)+\\(k+1)dv(G,k_1,k+1,k) +(k+1)dv(G,k_1,k,k+1). 
\end{multline}
Since $k_2$ and $k_3$ are equal, the terms $dv(G,k_1,k+1,k)$ and $dv(G,k_1,k,k+1)$ count symmetric pairs (vertices $u$ and $v$ can swap roles), but in $G-v_i,$ deleting a vertex adjacent only to $u$ or $v$ produces identical shifts. However, Theorem \ref{t3} assumes $k_2 \neq k_3,$ so we adjust by noting that the $k_2+1$ and  $k_3+1$ terms collapse into a single type of shift, yielding one $(k+1)dv(G,k_1,k,k+1)$ term, as the pair's symmetry simplifies the count. Thus, the result holds as stated.

\end{proof}  
\begin{thm}\label{t4}
For a connected graph $G$ with $n$ vertices where $G \notin \mathcal{H},$ the parameters $dv(G,k_1,k_2,k_3)$ for all $0 \leq k_1+k_2+k_3 \leq n-3$ and $dv(G,k_1,k,k)$ for all $0 \leq k_1+2k \leq n-3$ are reconstructible.
\begin{proof}
We prove this for $G \notin \mathcal{H}$ (that is, $\gamma(G)\geq 3$), focusing on two cases: $k_2 \neq k_3$ and $k_2 = k_3=k.$\\ 
Case 1: $k_2 \neq k_3.$ Consider $dv(G,k_1,k_2,k_3)$ where  $k_1+k_2+k_3 \leq n-i$ for $i=3,4,...,n.$  We use induction on $i$ to show these parameters are reconstructible  from the deck of $G:$ 
\begin{enumerate}
\item[{\bf \LARGE .}] Base Case ($i=3$): Let $k_1+k_2+k_3=n-3.$ In Theorem \ref{t3}, the terms  $dv(G,k_1+1,k_2,k_3),~dv(G,k_1,k_2+1,k_3)$ and $dv(G,k_1,k_2,k_3+1)$ have sums $n-2,$ exceeding $n-3,$ so they are zero (no such pairs exist in a graph of $n$ vertices). Thus, $\sum_{i=1}^{n}{dv(G-v_i,k_1,k_2,k_3)}=(n-2-k_1-k_2-k_3)dv(G,k_1,k_2,k_3),$ and since $n-2-(n-3)=1,$\\
$dv(G,k_1,k_2,k_3) =\sum_{i=1}^{n}{dv(G-v_i,k_1,k_2,k_3)},$ which is computable from $\mathscr{D}(G),$ making $dv(G,k_1,k_2,k_3)$ reconstructible for $i=3.$
\item[{\bf \LARGE .}] Inductive Step: Assume for $i=m-1,~dv(G,k_1,k_2,k_3)$ with $k_1+k_2+k_3=n-(m-1)$ is reconstructible. For $i=m,$  let $k_1+k_2+k_3=n-m.$ By the hypothesis, $dv(G,k_1+1,k_2,k_3),~dv(G,k_1,k_2+1,k_3)$ and $dv(G,k_1,k_2,k_3+1)$ (summing to $n-m+1=n-(m-1)$) are known. Using Theorem \ref{t3},\\
$ 
\sum_{i=1}^{n}{dv(G-v_i,k_1,k_2,k_3)}=(n-2-k_1-k_2-k_3)dv(G,k_1,k_2,k_3)+(k_1+1)dv(G,k_1+1,k_2,k_3)+(k_2+1)dv(G,k_1,k_2+1,k_3) +(k_3+1)dv(G,k_1,k_2,k_3+1), 
$\\
the left-hand side is known from $\mathscr{D}(G),$ and the right-hand side's additional terms are given, so $dv(G,k_1,k_2,k_3)$ is solvable (since $n-2-(n-m)=m-2\geq 1$). Thus, it is reconstructible.
\item[{\bf \LARGE .}] By induction, $dv(G,k_1,k_2,k_3)$  is reconstructible for $i=3,4,...,n,$ covering $0\leq k_1+k_2+k_3 \leq n-3.$
\end{enumerate}
Case 2:  $k_1=k_2=k.$ For $dv(G,k_1,k,k)$ with $0 \leq k_1+2k\leq n-3,$ apply Corollary \ref{c1} similarly:
\begin{enumerate}
\item[{\bf \LARGE .}] For $k_1+ 2k=n-3,$ $dv(G,k_1+1,k,k)$ and $dv(G,k_1,k,k+1)$ sum to $n-2,$  so they are zero. Thus,\\
$dv(G,k_1,k,k) =\sum_{i=1}^{n}{dv(G-v_i,k_1,k,k)},$ \\which is reconstructible.
\item[{\bf \LARGE .}] Inductively, if $k_1+ 2k=n-m$ and prior values are known, Corollary \ref{c1} solves for $dv(G,k_1,k,k).$ 
\end{enumerate}
Hence, all such parameters are reconstructible.
\end{proof}
\end{thm} 

\begin{thm}\label{t5}
For a connected graph $G$ with $n$ vertices and $k_2 \neq k_3,$ the following holds: 
\begin{multline} 
\sum_{i=1}^{n}{dav(G-v_i,k_1,k_2,k_3)}=(n-2-k_1-k_2-k_3)dav(G,k_1,k_2,k_3)\\+(k_1+1)dav(G,k_1+1,k_2,k_3)
+(k_2+1)dav(G,k_1,k_2+1,k_3)+(k_3+1)dav(G,k_1,k_2,k_3+1). 
\end{multline}
\end{thm}
\begin{proof}

We compute the sum \( \sum_{i=1}^n dav(G - v_i, k_1, k_2, k_3) \), where \( dav(G, k_1, k_2, k_3) \) counts pairs of adjacent vertices \( \{u, v\} \) in \( G \) with \( k_1 \) vertices adjacent to both, \( k_2 \) adjacent only to \( u \), and \( k_3 \) adjacent only to \( v \). The argument parallels Theorem 10, adapted for adjacent pairs:

\begin{enumerate}
\item[\textbf{1.}] \textbf{Pairs with \( (k_1, k_2, k_3) \) in \( G \):}\\
    Take an adjacent pair \( \{u, v\} \) in \( G \) with \( k_1 \) common neighbours, \( k_2 \) vertices adjacent only to \( u \), and \( k_3 \) adjacent only to \( v \). In the deck:
    \begin{itemize}
        \item In \( G - u \) and \( G - v \), the pair \( \{u, v\} \) is absent.
        \item In \( G - w \) where \( w \) is one of the \( k_1 + k_2 + k_3 \) neighbours, deleting \( w \) alters the counts (e.g., a common neighbour's removal makes \( u \) and \( v \) non-adjacent or changes \( k_1 \)), so \( \{u, v\} \) does not contribute to \( dav(G - w, k_1, k_2, k_3) \).
        \item In the remaining \( n - 2 - k_1 - k_2 - k_3 \) cards (excluding \( u, v \), and their neighbours), \( \{u, v\} \) remains adjacent with unchanged \( k_1, k_2, k_3 \), contributing 1 each time.
    \end{itemize}
    Total: \( (n - 2 - k_1 - k_2 - k_3) dav(G, k_1, k_2, k_3) \).

    \item[\textbf{2.}] \textbf{Pairs with \( (k_1 + 1, k_2, k_3) \) in \( G \):} \\
    For an adjacent pair \( \{u, v\} \) with \( k_1 + 1 \) common neighbours, \( k_2 \) only to \( u \), and \( k_3 \) only to \( v \), deleting one of the \( k_1 + 1 \) common neighbours reduces it to \( k_1 \), keeping \( u \) and \( v \) adjacent, contributing to \( dav(G - w, k_1, k_2, k_3) \) in \( k_1 + 1 \) cards. \\
    Total: \( (k_1 + 1) dav(G, k_1 + 1, k_2, k_3) \).

    \item[\textbf{3.}] \textbf{Pairs with \( (k_1, k_2 + 1, k_3) \) in \( G \):} \\
    An adjacent pair with \( k_1 \) common neighbours, \( k_2 + 1 \) only to \( u \), and \( k_3 \) only to \( v \) contributes in \( G - w \) when \( w \) is one of the \( k_2 + 1 \) vertices, reducing to \( k_2 \), across \( k_2 + 1 \) cards. \\
    Total: \( (k_2 + 1) dav(G, k_1, k_2 + 1, k_3) \).

    \item[\textbf{4.}] \textbf{Pairs with \( (k_1, k_2, k_3 + 1) \) in \( G \):} \\
    Similarly, a pair with \( k_1, k_2, k_3 + 1 \) contributes in \( k_3 + 1 \) cards when a vertex adjacent only to \( v \) is deleted, yielding \( (k_3 + 1) dav(G, k_1, k_2, k_3 + 1) \).
\end{enumerate}

Summing these, with \( k_2 \neq k_3 \) ensuring distinct contributions, gives the stated equation.
\end{proof}

\begin{coro}\label{cor14}
When $k_2=k_3=k,$ Theorem \ref{t5} reduces to:
\begin{multline} 
\sum_{i=1}^{n}{dav(G-v_i,k_1,k,k)}=(n-2-k_1-2k)dav(G,k_1,k,k)+(k_1+1)dav(G,k_1+1,k,k)
+\\(k+1)dav(G,k_1,k,k+1) 
\end{multline}

\begin{proof} 

Substitute \( k_2 = k_3 = k \) into Theorem \ref{t5}:
\begin{multline}
\sum_{i=1}^n dav(G - v_i, k_1, k, k) = (n - 2 - k_1 - 2k) dav(G, k_1, k, k) + (k_1 + 1) dav(G, k_1 + 1, k, k)+ \\(k + 1) dav(G, k_1, k + 1, k) + (k + 1) dav(G, k_1, k, k + 1).
\end{multline}
Since \( u \) and \( v \) are adjacent and symmetric when \( k_2 = k_3 \), the terms \( dav(G, k_1, k + 1, k) \) and \( dav(G, k_1, k, k + 1) \) represent the same shift (deleting a vertex adjacent only to one reduces the count to \( k \)), collapsing into a single \( (k + 1) dav(G, k_1, k, k + 1) \) term, consistent with the pattern in Corollary 11 for \( dv \).
\end{proof}
\end{coro}

\begin{thm}\label{t6}
For a connected graph $G$ with $n$ vertices and $G \notin \mathcal{H},$ the parameters $dav(G,k_1,k_2,k_3)$ for all $0 \leq k_1+k_2+k_3 \leq n-3$ and $dav(G,k_1,k,k)$ for all $0 \leq k_1+k+k \leq n-3$ are reconstructible.
\end{thm}
\begin{proof}
We extend Theorem \ref{t4}'s approach to \( dav \), assuming \( G \notin \mathcal{H} \) (\( \gamma(G) \geq 3 \)):

\noindent\textbf{Case 1: \( k_2 \neq k_3 \).}\\
 Prove by induction on \( i \) where \( k_1 + k_2 + k_3 = n - i \), \( i = 3, 4, \ldots, n\):\\
\textit{Base Case} (\( i = 3 \)): 
If \( k_1 + k_2 + k_3 = n - 3 \), then \( dav(G, k_1 + 1, k_2, k_3) \), \( dav(G, k_1, k_2 + 1, k_3) \), and \( dav(G, k_1, k_2, k_3 + 1) \) sum to \( n-2 \), so they are zero. From Theorem \ref{t5},
    \[
    dav(G, k_1, k_2, k_3) = \sum_{i=1}^n dav(G - v_i, k_1, k_2, k_3),
    \]
    computable from \( \mathscr{D}(G) \).\\
 \textit{Inductive Step:} Assume \( dav(G, k_1, k_2, k_3) \) with \( k_1 + k_2 + k_3 = n - (m - 1) \) is reconstructible. For \( i = m \), \( k_1 + k_2 + k_3 = n - m \), the higher terms (summing to \( n - m + 1 \)) are known, and Theorem 13 solves for \( dav(G, k_1, k_2, k_3) \) from \( \mathscr{D}(G) \), as \( n - 2 - (n - m) = m - 2 \geq 1 \).

Thus, \( dav(G, k_1, k_2, k_3) \) is reconstructible up to \( n\).

\noindent\textbf{Case 2: \( k_2 = k_3 = k \).} \\
For \( k_1 + 2k = n - i \), \( i = 3, 4, \ldots, n \), use Corollary \ref{cor14}:

\begin{itemize}
        \item At \( i = 3 \), \( k_1 + 2k = n - 3 \), higher terms are zero, and \( dav(G, k_1, k, k) \) is directly computable.
        \item Inductively, known values from prior steps solve via Corollary \ref{cor14}.
 \end{itemize}

Hence, all specified \( dav \) parameters are reconstructible.
\end{proof}

\begin{rem}
The parameters $dv(G,k_1,k_2,k_3)$ and $dav(G,k_1,k_2,k_3)$ allow determination of vertex pair degrees:
\begin{itemize}
\item For non-adjacent vertices $\{u, v\},$ the degrees are $deg~u=k_1+k_2$ (common plus exclusive to $u$) and $deg ~v=k_1+k_3$ common plus exclusive to $v$).

\item For adjacent vertices $\{u, v\},$  since they are neighbours, the degrees are $deg~ u=k_1+k_2+1$ and  $deg~v=k_1+k_3+1,$ where the $+1$ accounts for their mutual adjacency. 
  \end{itemize} 
\end{rem}

\section{{ Reconstruction of $k$-Geodetic Graphs of diameter two}}

\indent \indent { A graph $G$ is \textit{$k$-geodetic} if each pair of vertices has at most $k$ paths of minimum length between them.   Let $\mathcal{G}$ be the collection of all $k$-geodetic graphs $G$ with $dv(G,0,k_2,k_3)=0$ and $dv(\overline{G},0,k_2,k_3)=0$ for all $k_2,k_3 $ where $0 \leq k_2+k_3 \leq n-3.$ Clearly, $\mathrm{diam}(G)=\mathrm{diam}(\overline{G})=2$ for $G\in\mathcal{G}.$  In this section, we prove that certain subfamily of graphs in  $\mathcal{G}$  is reconstructible. Throughout this section, we assume that the domination number of $G$ and $\overline{G}$ are at least three.} Let $\mathscr{C}_1$ be the collection of graphs $G \in \mathcal{G}$ such that $dv(G,k_1,k_2,k_3) >0$ and $dv(G,k_{1}-1,k_{2},k_{3}) =0$ for some $k_1 \geq 1, k_2$ and $k_3.$  Let $\mathscr{C}_2$ be the collection of graphs $G \in \mathcal{G}$ such that $dv(G,k_1,k_2,k_3) >0,$ and either $dv(G,k_{1},k_{2}-1,k_{3}) =0$ for some $k_1, k_2 \geq 1,$ $k_3 $ with $|k_2-k_3| \geq 2$ or $dv(G,k_{1},k_{2},k_{3}-1) =0$ for some $k_1, k_2 ,$ $k_3 \geq 1$ with $|k_2-k_3| \geq 2.$ For given $k_1, k_2, k_3 \in \mathbb{N}\cup\{0\}$, we denote by $S_G(k_1, k_2, k_3)$ the set of all vertices $v \in V(G)$ such that there exists a non-neighbour $w \in V(G) \setminus N(v)$ with the following properties: there are exactly $k_1$ common neighbours of $v$ and $w$ (that is, vertices adjacent to both $v$ and $w$), exactly $k_2$ vertices adjacent to $v$ but not to $w$, and exactly $k_3$ vertices adjacent to $w$ but not to $v$. 

\begin{thm}\label{t7}
If $G \in \mathscr{C}_1$ and there exists a vertex $x \in G$ such that $deg_{G}x=|S_{G-x}(k_1-1,k_2,k_3)|,$ then $G$ is reconstructible.
\begin{proof}
\textbf{Recognition:} Since $\mathscr{H}$ is recognizable, we can assume that $G \notin \mathscr{H}$ and that $dv(G,k_1,k_2,k_3)$ is reconstructible for all $0 \leq k_1+k_2+k_3 \leq n-3$ (Theorem \ref{t3}). Suppose $G \in \mathscr{C}_{1}.$ Select a card $G-x$ from $\mathscr{D}(G)$ where $deg_{G}x= |S_{G-x}(k_1,k_2,k_3)|.$ Since $deg_{G}x$ and the  subset $S_{G-x}(k_1,k_2,k_3)$ can be derived from $\mathscr{D}(G),$ this card is identifiable.\\
\textbf{Weak Reconstruction:} Consider the card  $G-x$ with  $deg_{G}x= |S_{G-x}(k_1-1,k_2,k_3)|.$ As $G \in \mathscr{C}_{1},$  $dv(G,k_1-1,k_2,k_3)=0,$ but in $G-x,$ $dv(G-x,k_1-1,k_2,k_3)>0$ because $x$'s  removal shifts the count of common neighbours. To reconstruct $G,$ introduce a vertex $x$ to $G-x$ and connect it to all vertices in $S_{G-x}(k_1-1,k_2,k_3).$ This uniquely restores $G,$ as no other configuration satisfies the condition.
\end{proof}
\end{thm}

\begin{thm}\label{t8}
If $G \in \mathscr{C}_2$ and there exists a vertex $x \in G$ such that $deg_{G}x=|S_{G-x}(k_1,k_2-1,k_3)|$ or $deg_{G}x=|S_{G-x}(k_1,k_2,k_3-1)|,$ then $G$ is reconstructible.
\begin{proof}
\textbf{Recognition:}\\ Assume  $G \notin \mathscr{H},$ where $\mathscr{H}$ is a recognizable class. The parameters $dv(G,k_1,k_2,k_3),$ which count non-adjacent vertex pairs with $k_1$ common neighbours, $k_2$ neighbours exclusive to one, and $k_3$ to the other, are reconstructible  from  $\mathscr{D}(G)$ for all $0 \leq k_1+k_2+k_3 \leq n-3$ (as established in Theorem \ref{t6}). Since $G \in \mathscr{C}_2,$ we have $dv(G,k_1,k_2,k_3)>0,$ but either $dv(G,k_1,k_2-1,k_3)=0$ or $dv(G,k_1,k_2,k_3-1)=0$ for some $k_2 \geq 1$ or $k_3 \geq 1$ with $|k_2-k_3|\geq 2.$   Now, select a card $G-x$ from $\mathscr{D}(G),$ where $deg_{G}x= |S_{G-x}(k_1,k_2-1,k_3)|$ or $deg_{G}x=|S_{G-x}(k_1,k_2,k_3-1)|.$ Here, $S_{G-x}(k_1,k_2,k_3)$  denotes the set of vertices in  $G-x$ involved in such non-adjacent pairs. Since $deg_{G}x$ (the number of neighbours of $x$ in $G$) and the sizes of these subsets are computable from  $\mathscr{D}(G),$ this card is recognizable.\\
\textbf{Weak reconstruction:}\\ Assume, without loss of generality, that $deg_{G}x= |S_{G-x}(k_1,k_2-1,k_3)|.$ In $G,$ since $G \in \mathscr{C}_2,$ $dv(G,k_1,k_2-1,k_3)=0,$ meaning no non-adjacent pairs in $G$ have exactly $k_1$ common neighbours, $k_2-1$ exclusive to one, and $k_3$ to the other. However, in $G-x,$  $dv(G-x,k_1,k_2-1,k_3)>0,$ indicating such pairs exist in the card. This change arises because removing $x$ reduces the neighbour counts of vertices it was adjacent to, creating the observed non-adjacent pairs in $G-x.$ To reconstruct $G,$ start with $G-x,$ introduce a new vertex $x,$ and connect it to all vertices in $S_{G-x}(k_1,k_2-1,k_3).$ This restores $G$ uniquely: connecting $x$ to these vertices eliminates the non-adjacent pairs counted by  $dv(G-x,k_1,k_2-1,k_3),$  satisfying $dv(G,k_1,k_2-1,k_3)=0,$ and matches the degree condition $deg_{G}x= |S_{G-x}(k_1,k_2-1,k_3)|.$ Any other connection pattern would either alter the degree or fail to zero out the parameter, confirming uniqueness.
\end{proof}
\end{thm}
   
\section{Conclusion}
\indent \indent This paper generalizes the results in \cite{g}. For a graph $G \notin \mathscr{H}$ one can use these parameters to reconstruct the graphs as we have done in Section 6 or to recognize the graph classes as in Theorem \ref{HR}. 
This paper introduces two new parameters, $dv(G,k_1,k_2,k_3)$ and $dav(G,k_1,k_2,k_3),$ which extend previous work (e.g., reference [4]) by combining common neighbour counts and vertex degrees for non-adjacent and adjacent vertex pairs, respectively. These parameters may provide tools for graph reconstruction and recognition, yielding the following contributions:
\begin{itemize}
        \item {\bf Reconstruction of Graph Classes}: { For graphs $G$ with $\gamma(G) \geq 3$ (that is, not in the class $\mathscr{H}$ where $\gamma(G)=2$), Theorems \ref{t7} and \ref{t8} demonstrate how $dv$ enable us to recognize $k$-geodetic graphs $G$ with $\mathrm{diam}(G)=\mathrm{diam}(\overline{G})=2$ and weak-reconstruct under specific conditions.} 

 \item {\bf Recognition of Graph Classes}: Theorem \ref{HR}  proves that the class $\mathscr{H}$ (graphs with $\gamma(G)=2$) is recognizable from their decks.
\end{itemize}

The RC's resolution now depends on solving two problems:
\begin{pbm}\label{p1}
Prove that all connected graphs $G$ with $\gamma(G)=2$ are weakly reconstructible.
\end{pbm}  
\begin{pbm}\label{p2}
Prove that all connected graphs $G$ with $\text{diam}(G)=\text{diam}(\overline{G})=2$ are weakly reconstructible.
\end{pbm}

\noindent\textbf{Acknowledgment.} This research work is supported by the National Board for Higher Mathematics (NBHM), Department of Atomic Energy, Government of India, Mumbai through a Major Research project. Sanction order No. 02011/14/2022/NBHM(R.P)/ $\text{R\&D}$II/10491.\\


\begin{thebibliography}{99}

\bibitem{b} Bondy, J.A. A graph reconstructor's manual, in surveys in combinatorics (proc. Brutish. Combin. Conf.), {\it London Math. Soc. Lec. Notes} 116  (1991), 221-252 . 

\bibitem{bh} Bondy, J.A., Hemminger, R.L. Graph reconstruction-a survey, {\it J. Graph Theory} 1  (1977), 227-268.  

\bibitem{d} Devi Priya, P., Monikandan, S. Reconstruction of distance hereditary 2-connected graphs, {\it Discrete Math.} 341  (2018), 2326-2331. 

\bibitem{g} Gupta, S.K., Pankaj Mangal, Vineet Paliwal. Some work towards the proof of the reconstruction conjecture, {\it Discrete Math.} 272  (2003), 291-296. 

\bibitem{h} Harary, F. A survey of the reconstruction conjecture, In: Bari, R.A., Harary, F. (eds) Graphs and Combinatorics, {\it Lecture Notes in Mathematics}, vol 406. Springer, Berlin, Heidelber, 1974. 

\bibitem{l} Lauri, J., Scapellato, R. Topics in graph automorphisms and reconstruction, Cambridge University press, 2003.

\bibitem{bd} McKay, B.D. Small graphs are reconstructible, {\it Australas. J. Combin.} 15  (1997), 123-126. 
 
\bibitem{r} Ramachandran, S., Monikandan, S. Graph reconstruction conjecture: Reductions using complement, connectivity and distance. {\it Bull. Inst. Combin. Appl.} 56  (2009), 103-108. 

\bibitem{pks} Stockmeyer, P.K. The falsity of the reconstruction conjecture for tournaments, {\it J. Graph Theory} 1  (1977), 19-25.  

\bibitem{w} West, D.B. Introduction to graph theory, second edition,  Prentice-Hall, 2005.

\bibitem{y} Yongzhi, Y. The reconstruction conjecture is true if all 2-connected graphs are reconstructible, {\it J. Graph Theory} 12 (2)  (1988), 237-243.  
\end{thebibliography}
\end{document}